\documentclass[reqno,11pt]{amsart}

\usepackage{xcolor}
\usepackage{graphicx}
\usepackage[hidelinks]{hyperref}

\usepackage{latexsym,array}
\usepackage{amsfonts}
\usepackage{shadow}
\usepackage{amsbsy}
\usepackage{amssymb}
\usepackage{dsfont}
\usepackage{mathtools}
\usepackage{enumitem}
\usepackage{doi}

\usepackage[notref, notcite, final]{showkeys}
\usepackage{fullpage}
\usepackage{comment}
\usepackage[latin1]{inputenc}

\usepackage{cleveref}

\numberwithin{equation}{section}

\newcommand{\id}{\mathcal{I}}

\newcommand{\bft}{\mathbf{t}}
\newcommand{\bfb}{\mathbf{b}}

\newcommand{\dom}{\mathcal{D}}

\theoremstyle{theorem}
\newtheorem{theorem}{Theorem}[section]

\newtheorem{proposition}[theorem]{Proposition}

\theoremstyle{definition}
\newtheorem{remark}[theorem]{{\bf Remark}}
\newtheorem{definition}[theorem]{Definition}
\newtheorem{problem}[theorem]{Problem}

\newcommand{\cc}{\mathbb{C}}

\newcommand{\hh}{\mathbb{H}}

\newcommand{\rr}{\mathbb{R}}

\newcommand{\boundOP}{\mathcal{B}}
\newcommand{\closOP}{\mathcal{K}}

\newcommand{\vx}{{{x}}}

\newcommand{\Q}{\mathcal{Q}}

\renewcommand{\Re}{\mathrm{Re}}

\newcommand{\uI}{j}

\crefname{enumi}{}{}
\crefname{enumii}{}{}

\newcommand{\CB}{\color{black}}

\title[]
{Fractional powers of higher order vector operators on bounded and unbounded domains}

\author[L. Baracoo]{Luca Baracco}
\address{(LB)
	Universit\'a di Padova\\ Dipartimento di Matematica\\ Via Trieste, 63\\ Padova, Italy}
\email{baracco@unipd.it}

\author[F. Colombo]{Fabrizio Colombo}
\address{(FC)
Politecnico di Milano\\Dipartimento di Matematica\\Via E. Bonardi, 9\\20133
Milano, Italy}
\email{fabrizio.colombo@polimi.it}

\author[M. M. Peloso]{Marco M. Peloso}
\address{(MP)
Universit\'a degii studi di Milano\\ Dipartimento di Matematica\\Via Saladini, 50\\20133
Milano, Italy}
\email{Marco.Peloso@unimi.it}

\author[S. Pinton]{Stefano Pinton}
\address{(SP)
Politecnico di Milano\\Dipartimento di Matematica\\Via E. Bonardi, 9\\20133
Milano, Italy}
\email{stefano.pinton@polimi.it}

\begin{document}

\maketitle

\begin{abstract}
Using \CB the $H^\infty$-functional calculus for quaternionic operators,
we show how to generate the fractional powers of some densely defined
differential quaternionic operators of order $m\geq 1$, acting on the
right linear quaternionic Hilbert space
$L^2(\Omega,\cc\otimes\hh)$. The operators that we consider are of the
type 
$$
T=i^{m-1}\left(a_1(x) e_1\partial_{x_1}^{m}+a_2(x)
  e_2\partial_{x_2}^{m}+a_3(x) e_3\partial_{x_3}^{m}\right), \ \ \
x=(x_1,\, x_2,\, x_3)\in \overline{\Omega}, 
$$
where $\overline{\Omega}$ is
the closure of either a bounded domain $\Omega$ with $C^1$
boundary, or 
an unbounded domain
 $\Omega$ in $\rr^3$ with a sufficiently regular boundary which
 satisfy the so called property $(R)$ (see Definition \ref{pR}),
 $\{e_1,\, e_2,\, e_3\}$ is an orthonormal basis for the imaginary
 units of $\hh$, $a_1,\,a_2,\, a_3: \overline{\Omega}
 \subset\mathbb{R}^3\to \mathbb{R}$ are the coefficients of $T$. In
 particular it will be given sufficient conditions on the coefficients
 of $T$ in order to generate the fractional powers of $T$, denoted by
 $P_{\alpha}(T)$ for $\alpha\in(0,1)$, when the components of $T$,
 i.e. the operators $T_l:=a_l\partial_{x_l}^m$, do not commute among
 themselves.

\end{abstract}
\vskip 1cm
\par\noindent
 AMS Classification: 47A10, 47A60.
\par\noindent
\noindent {\em Key words}:  Fractional  powers, higher order vector operators,
$S$-spectrum, $S$-spectrum approach.
\vskip 1cm

\section{\bf Introduction }
Using the $S$-functional
calculus, in the series of papers   \cite{CGLax}, \cite{CMPP}, \cite{CPP},
\cite{JGPFRAC}, \cite{CDP20}, we defined the fractional powers of a class of vector operators with
 non-constant coefficients. 
In this paper we consider the quaternionic differential
operators of the form
$$ 
T=i^{m-1}\left(a_1(x) e_1\partial_{x_1}^{m}+a_2(x)
  e_2\partial_{x_2}^{m}+a_3(x) e_3\partial_{x_3}^{m}\right), \ \ \
x=(x_1,\, x_2,\, x_3)\in \overline{\Omega}, 
$$
and we prove that under suitable conditions on the
coefficients it admits well defined  fractional powers. In order to state our results we give
some details on the quaternion techniques based on the spectral theory
on the $S$-spectrum.  For a complete introduction of the
$S$-functional calculus see the books \cite{FJBOOK}, \cite{CGKBOOK},
here we briefly introduce the main aspects of this theory. 

\subsection{\bf The $S$-functional calculus}
An element in the quaternions $\mathbb{H}$ is of
 the form  $s=s_0+s_1e_1+s_2e_2+s_3e_3$,  where $s_0$, $s_\ell$ are real numbers ($\ell=1,2,3$), ${\rm Re}(s):=s_0$ denotes the real part of $s$ and $e_\ell$, for $\ell=1,2,3$, are the imaginary units which satisfy the relations: $e_1^2=e_2^2=e_3^2=e_1e_2e_3=-1$. The modulus
    of $s$ is defined as $|s|=(s_0^2+s_1^2+s_2^2+s_3^2)^{1/2}$
    and the conjugate is given by $\overline{s}=s_0-s_1e_1-s_2e_2-s_3e_3$.
    In the sequel we will denote by $\mathbb{S}$ the unit sphere of
    purely imaginary quaternions, an element $j$ in $\mathbb{S}$ is
    such that $j^2=-1$. We consider a two-sided quaternionic Banach space $V$ and
we denote the set of closed densely defined quaternionic right
linear operators on $V$ by 
 $\closOP(V)$.
  The Banach space of all bounded right linear operators on $V$ is indicated by the symbol $\mathcal{B}(V)$ and is endowed with the natural operator norm.
For $T\in\closOP(V)$, we define the operator associated with the $S$-spectrum as:
\begin{equation}\label{QST}
\Q_{s}(T) := T^2 - 2\Re(s)T + |s|^2\id, \qquad \text{for $s\in\hh$}
\end{equation}
where $\Q_{s}(T):\mathcal{D}(T^2)\to V$, where $\mathcal{D}(T^2)$ is the domain of $T^2$.
 We define the $S$-resolvent set of  $T$ as
\[\rho_S(T):= \{ s\in\hh: \Q_{s}(T) \ {\rm is\ invertible\ and \ } \Q_{s}(T)^{-1}\in\boundOP(V)\}\]
and the $S$-spectrum of $T$ as
\[\sigma_S(T):=\hh\setminus\rho_S(T).\]
The operator $\Q_{s}(T)^{-1}$ is called the pseudo $S$-resolvent operator.
 For $s\in\rho_S(T)$, the left $S$-resolvent operator is defined as
\begin{equation}\label{SRESL}
S_L^{-1}(s,T):= \Q_s(T)^{-1}\overline{s} -T\Q_s(T)^{-1}
\end{equation}
and the right $S$-resolvent operator is given by
\begin{equation}\label{SRESR}
S_R^{-1}(s,T):=-(T-\id \overline{s})\Q_s(T)^{-1}.
\end{equation}
The fractional powers of an operator $T$ such that $j\mathbb
R\subset\rho_S(T)$ for any $j\in \mathbb S$, are denoted by
$P_{\alpha}(T)$ and are defined as follows. 
For  any $j\in \mathbb{S}$,  for $\alpha\in(0,1)$ and  $v\in\dom(T)$ we set
\begin{equation}\label{BALA1}
P_{\alpha}(T)v := \frac{1}{2\pi} \int_{-j\rr}   S_L^{-1}(s,T)\,ds_j\, s^{\alpha-1} T v,
\end{equation}
or
\begin{equation}\label{BALA2}
P_{\alpha}(T)v := \frac{1}{2\pi} \int_{-j\rr} s^{\alpha-1} \,ds_j\,  S_R^{-1}(s,T) T v,
\end{equation}
where $ds_j=ds/j$.
These formulas are a consequence
of the quaternionic version of the $H^\infty$-functional calculus
based on the $S$-spectrum, again see   \cite{FJBOOK} for details.
For the generation of the fractional powers $P_{\alpha}(T)$, a crucial assumption
on the $S$-resolvent operators is that, for $s\in \mathbb{H}\setminus \{0\}$ with ${\rm Re}(s)=0$, the estimates
\begin{equation}\label{SREST}
\left\|S_L^{-1}(s,T)\right\|_{\mathcal{B}(V)} \leq \frac{\Theta}{|s|}\quad\text{and}\quad
\left\|S_R^{-1}(s,T)\right\|_{\mathcal{B}(V)} \leq \frac{\Theta}{|s|},
\end{equation}
hold
with a constant $\Theta >0$ that does not depend on $s$.
 It is important to observe that
the conditions (\ref{SREST}) assure that the integrals (\ref{BALA1}) and (\ref{BALA2}) are convergent and so the fractional powers are well defined.

For the definition of the fractional powers of the operator $T$ we can use equivalently the integral representation in (\ref{BALA1}) or the one in (\ref{BALA2}).
Moreover, they correspond to a modified version of Balakrishnan's formula that takes only spectral points with positive real part into account.
\begin{remark}
It is clear from the definition of the $S$-resolvent operators that to
use the $S$-functional calculus for the definition of the
fractional powers of an operator $T$ we have to determine if $\mathcal
Q_s(T)$ is invertible for any $s\in\hh$ such that $s\neq 0$ and
$\Re(s)=0$, and, moreover, if estimates of the type \eqref{SREST}
hold, see Problem \ref{ProbAA}. 
\end{remark}

\subsection{\bf Operators of order one vs operators of order $m>1$} In some of our previous papers, we have defined fractional powers of operators of first order such as
$$
T:=\begin{pmatrix}
a_1(x)\partial_{x_1}\\a_2(x)\partial_{x_2}\\a_3(x)\partial_{x_3}.
\end{pmatrix}
$$
acting on functions $u:\Omega\subset\rr^3\to\rr$ belonging to
$H^1_0(\Omega,\rr)\subset L^2(\Omega,\rr)$ where $\Omega$ is a
(possibly) unbounded domain with $C^1$ boundary. To use the $S$-functional calculus, first we identify the gradient operator with
the quaternionic gradient operator  
$$
\begin{pmatrix}
a_1(x)\partial_{x_1}\\a_2(x)\partial_{x_2}\\a_3(x)\partial_{x_3}
\end{pmatrix}\equiv e_1 a_1(x)\partial_{x_1}+e_2a_2(x)\partial_{x_2}+e_3a_3(x)\partial_{x_3}
$$
and we consider the operator $\mathcal Q_{s}(T)$ defined in a weak
sense over $H^1_0(\Omega,\hh)\subset L^2(\Omega, \hh)$. It is
important to observe that the above identification has 
some 
consequence on the bilinear form $b_s(u,v)=\langle \mathcal Q_s(T)(u),
v\rangle_\hh$, where  $\langle a,b \rangle_\hh:=\overline a b$ for all
$a,\, b\in\hh$. Indeed, performing an integration by parts we have 
 (see in the following):  
$$
b_s(u,u)=\sum_{l=1}^3\|a_l(x)\partial_{x_l}u\|^2+|s|^2\|u\|^2+\textrm{ other terms}
$$
where the "other terms" are the scalar products of the first derivatives of $u$ with $u$ multiplied by the derivative of the coefficients.
 In the  formula, we indicate that $b_s(u,u)$ contains two
positive terms: the $L^2$-norm of $u$ and the $L^2$-norm of the first
derivative of $u$ multiplied by the coefficients $a_j$'s. 
This fact allows us to determine
some conditions on the coefficients $a_j$'s in order to guarantee the
continuity and the coercivity of $b_s(\cdot,\cdot)$ and, moreover, the
uniform estimates for the $S$-resolvent operator.

In this
paper we consider vector operators of order $m>1$ and $m\in\mathbb N$
of the type 
$$ T:=\begin{pmatrix}
a_1(x)\partial_{x_1}^m\\a_2(x)\partial_{x_2}^m\\a_3(x)\partial_{x_3}^m
\end{pmatrix}.
$$
If we consider the previous identification
$$
T 
\equiv e_1
a_1(x)\partial_{x_1}^m+e_2a_2(x)\partial_{x_2}^m+e_3a_3(x)\partial_{x_3}^m
\, ,
$$
and we consider the operator $\mathcal Q_s(T)$ defined in a weak sense
over $H^m_0(\Omega,\hh)\subset L^2(\Omega, \hh)$, we have to
distinguish the cases of $m$ odd or even. If $m$ is odd using
$m$-times an argument of integration by parts we obtain 
$$b_s(u,u)=\sum_{l=1}^3\|a_l(x)\partial_{x_l}^mu\|^2+|s|^2\|u\|^2+\textrm{ other terms}$$
and the bilinear form still contains two positive terms: the
$L^2$-norm of $u$ and the $L^2$-norm of the derivatives of order
$m$. If we try to compute the bilinear form in the same way when $m$
is even, we obtain 
$$ b_s(u,u)=-\sum_{l=1}^3\|a_l(x)\partial_{x_l}^mu\|^2+|s|^2\|u\|^2+\textrm{ other terms} $$
losing the positivity of the first term. One way to overcome this problem is to identify $T$ with
$$e_1 a_1(x)\partial_{x_1}^m+e_2a_2(x)\partial_{x_2}^m+e_3a_3(x)\partial_{x_3}^m$$
if $m$ is odd and with
$$i(e_1 a_1(x)\partial_{x_1}^m+e_2a_2(x)\partial_{x_3}^m+e_3a_3(x)\partial_{x_3}^m)$$
if $m$ is even where $i$ is the imaginary unit of $\cc$. In other
words, we have complexified the coefficients of the quaternionic
gradient operator and the operator $T$ is identified with the
quaternionic gradient operator with real coefficients if $m$ is odd or
with the quaternionic gradient operator with purely imaginary
coefficients if $m$ is even. For the precise definitions of
$b_s(\cdot,\cdot)$ and of the quaternionic Hilbert space
$\cc\otimes\hh$, see Section \ref{SSECdue}.  In light of these
considerations we give the following definition. 
\begin{definition}\label{TCOM} Let
$\Omega$ be a $C^1$-domain in $\rr^3$, bounded or unbounded, and
let 
  $a_l: \overline\Omega\to\rr$ for $l=1,\, 2,\, 3$ be $C^m(\overline\Omega)$ functions. We define in a classical way over $C^m(\overline\Omega, \cc\otimes\hh)$ the operator
$$
T:=i^{m-1}\left(a_1(x) e_1\partial_{x_1}^{m}+a_2(x) e_2\partial_{x_2}^{m}+a_3(x) e_3\partial_{x_3}^{m}\right), \ \ \ x=(x_1,\, x_2,\, x_3)\in \overline{\Omega}.
$$
\end{definition}
\subsection{\bf $\Omega$ bounded vs $\Omega$ unbounded} We will treat
separately the cases of $\Omega$ bounded and of $\Omega$
unbounded. The unbounded case is more complicated as explained at the
end of this section and needs some more constraints on the shape of
$\Omega$ that we now introduce. For the following definition, we shall
utilize $n$-dimensional spherical coordinates $(r,\omega)$ where
$r\geq 0$ is the distance from the origin, $\omega\in S_{n-1}$,
and $ S_{n-1}$ denotes the sphere in $\rr^n$ (see \cite{EM80}; in our
case however, $n=3$).
\begin{definition}\label{pR}
An open set $\Omega\subset \rr^3$ is said to have the property $(R)$
if there exists $P\in\rr^3\setminus \overline\Omega$ such that every
ray through $P$ has intersection with $\Omega$ which is either empty
or an infinite interval. More precisely, for each $\omega\in S_2$
 set
 $$
 \begin{cases}
 r(\omega):=\inf \{r\geq 0:\, P+r\omega\in\Omega\} &\textrm{if}\{P+r\omega:r\geq 0\}\cap \Omega\neq \emptyset\\
 r(\omega):=\infty & \textrm{if} \{P+r\omega:r\geq 0\}\cap \Omega= \emptyset.
 \end{cases}
$$
 We are assuming that if $r(\omega)\neq\infty$, then $P+r\omega\in\Omega$ for all $r\in(r(\omega),\infty)$.
  \end{definition}
  Examples of unbounded domains which satisfies the property $(R)$ are: $\Omega:=\{x\in\rr^3:\, |x-P|>M\}$ and $\Omega:=\{x\in\rr^3 :\, \langle P-x,v\rangle>0\}$ where $v\in\rr^3$ is a vector, $P\in\rr^3$ is a point and $M>0$ is a positive constant (here $\langle \cdot,\cdot\rangle$ is the standard scalar product of $\rr^3$). We are ready to formulate in the precise way the problems that we have to solve.
\begin{problem}\label{ProbAA}
Let $\Omega\subset\rr^3$ be with $C^1$ boundary, which is either bounded
or unbounded  and satisfying the property $(R)$.
Let $F:\Omega\to\cc\otimes\hh$ be a given $L^2$-function and denote by
$u:\Omega\to\cc\otimes\hh$ the unknown function satisfying the
boundary value problem: 
\begin{equation}\label{eqprob1}
\begin{cases}
\mathcal Q_{s}(T) (u)=F\\
\partial^{\bfb} u(x)=0 & \forall \, \bfb\in\mathbb N^3_0\,\textrm{such
  that}\, |\bfb|\leq m-1\,{\rm and}\, x\in \partial\Omega,
\end{cases}
\end{equation}
where ${\bf b} =(b_1,b_2,b_3)$ and  $\partial^{\bfb}=\partial^{b_1}_{x_1}\partial^{b_2}_{x_2}\partial^{b_3}_{x_3}$.
Determine the conditions on the coefficients $a_1, a_2, a_3 : \Omega
\to\rr$ such that the boundary value problem has a unique solution in
a suitable function space and, moreover, the $L^2$-estimates
\eqref{SREST} for the $S$-resolvent operators hold.
\end{problem}
\begin{remark}
If $\Omega$ is bounded and its boundary is of class $C^m$, the previous boundary value problem is equivalent to
\begin{equation}\label{eqprob2}
\begin{cases}
\mathcal Q_{s}(T) (u)=F\\
\partial_{\nu}^j u(x)=0 & \textrm{for $0\leq j\leq m-1$ and}\, x\in \partial\Omega, 
\end{cases}
\end{equation}
where, $\nu$ is the normal vector field pointing outside $\partial\Omega$ (see Theorem 7.41 in \cite{RRBOOK}).
\end{remark}
We are going to solve the previous problem by the use of the
Lax-Milgram Lemma. In particular, to solve the boundary value problem
\eqref{eqprob1}, we want to find some conditions on the coefficients
$a_j$'s in a such way that the continuity and the coercivity of the
quadratic form $b_s(u,v)$ associated to $\mathcal Q_s(T)$ hold (see
Definition \ref{weakformbis}). The need of proving also the
estimates \eqref{SREST} makes the assumptions on the coefficients
$a_l$'s stronger than the usual one that we have to require for the
coercivity of $b_s(\cdot,\cdot)$, since we can not relay on the term
$|s|^2\|u\|^2$ of $b_s(\cdot,\cdot)$. For this reason, the other
positive term in $b_s(u,u)$, that is $\sum_{l=1}^3\|a_l\partial_{x_l}^m u\|$,
has to control the $L^2$-norm of $u$ and the $L^2$-norms of all the
partial derivatives of $u$ up to order $m-1$. In Section \ref{s3b},
when $\Omega$ is a bounded domain of $\rr^3$, through an iterated use
of  Poincar\'e's inequality we will show that the conditions: \begin{itemize}
\item $|a_j| \gg\max\left(C_\Omega^m, C_\Omega\right)$;
  \item $|a_j| \gg
    |\partial^\beta a_j|$ for any $|\beta|<m$,
  \end{itemize}
  are sufficient to solve the
  Problem \ref{ProbAA}. Here and in what follows,
  $C_\Omega$ denotes the Poincar\'e constant of
$\Omega$). When $\Omega$ is an unbounded domain of $\rr^3$ and $m=1$,
the role of the Poincar\'e inequalities is replaced by the
Gagliardo-Nirenberg estimates and a condition of integrability on the
first derivatives of the coefficients is sufficient to get the coercivity of $b_s(\cdot,\cdot)$ (see \cite{CDP20}). 

When $\Omega$
is an unbounded domain of $\rr^3$ and the order of the operator $T$ is
greater than $1$, the Gagliardo-Nirenberg estimates can not be used in
an iterated way as for Poincar\'e's inequality. In Section
\ref{SSECdue}, we propose one way to overcome this problem
by the use of a weighted Poincar\'e's inequality on some unbounded
domains under  an exponential decay condition at infinity of the
coefficients $a_j$'s (see Theorem \ref{t3}). 
\section{\bf The weak formulation of the Problem \ref{ProbAA}}\label{SSECdue}

 The boundary $\partial\Omega$ of $\Omega$ is assumed to be of class $C^1$ even though
for some lemmas in the sequel the conditions on the open set $\Omega$
can be weakened. We consider the right quaternionic Hilbert
space 
$\cc\otimes\hh$ endowed with the scalar product 
\[
\langle u,v \rangle:= \overline q_1 w_1+\overline q_2 w_2\quad\forall\, u,\, v\in\cc\otimes\hh
\]
where $i$ is the imaginary unit of $\cc$, $u=q_1+i q_2$, $v=w_1+iw_2$
and $q_1,\, q_2,\, w_1,\, w_2\in\hh$. As usual, we define the modulus
of $u\in\cc\otimes\hh$ as 
$$
|u|:=\sqrt{\langle u,u \rangle}.
$$
We observe that a function $u:\Omega\to\cc\otimes\hh$ is determined by $8$ real functions $u_{j,l}:\Omega\to \rr$ where $j= 0,\, 1,\, 2,\, 3$ and $l=1,\, 2$. We call these functions the components of $u$. We will use the following notation
\[
\begin{split}
u(x)&=(u_{0,1}(x)+u_{1,1}(x) e_1+u_{2,1}(x)e_2+u_{3,1}(x)e_3)+i(u_{0,2}(x)+u_{1,2}(x)e_1+u_{2,2}(x)e_2+u_{3,2}(x)e_3)
\\
&
=u_1(x)+iu_2(x).
\end{split}
\]
where $u_1(x):=u_{0,1}(x)+u_{1,1}(x)e_1+u_{2,1}(x)e_2+u_{3,1}(x)e_3$ and $u_2(x):=u_{0,2}(x)+u_{1,2}(x)e_1+u_{2,2}(x)e_2+u_{3,2}(x)e_3$. We can consider the space of $L^p$-integrable functions from a domain $\Omega\subset \rr^3$ to $\cc\otimes\hh$
\[
L^p:= L^p(\Omega,\cc\otimes\hh) := \left\{u: \Omega\to\cc\otimes\hh: \int_{\Omega}|u(\vx)|^p\,d\vx < + \infty\right\}.
\]
The space $L^2$ with  the scalar product:
$$
\langle u,v\rangle_{L^2} := \langle u,v\rangle_{L^2(\Omega,\cc\otimes\hh)} := \int_{\Omega} \langle u(\vx),v(\vx)\rangle\,d\vx \quad \forall u,\, v\in L^2(\Omega,\cc\otimes\hh),
$$
 is a right linear quaternionic Hilbert space.
 We furthermore introduce the quaternionic Sobolev space of order $m$
 \[
 H^m:= H^m(\Omega,\cc\otimes\hh): = \Big\{u\in L^2(\Omega,\cc\otimes\hh) : \, u_{j,l}\in H^m(\Omega, \rr)\quad j=0,1,2,3\,\textrm{and}\, l=1,2\Big\},
 \]
 where the space $H^m(\Omega,\rr)$ is the Sobolev space of order $m$ defined as in \cite{Evans} Chapter $5$. We have that $ H^m(\Omega,\cc\otimes\hh)$ endowed with the quaternionic scalar product
 \[
 \langle u,v\rangle_{H^m} := \langle u, v \rangle_{H^m(\Omega,\hh)} := \langle u,v\rangle_{L^2} + \sum_{1\leq |{\bfb}|\leq m}^3 \left\langle \partial^{\bfb} u,\partial^{\bfb} v\right\rangle_{L^2},
 \]
where  $\bfb\in\mathbb N^3$, becomes a right linear quaternionic Hilbert space.
 As usual the space $H^m_0(\Omega, \cc\otimes\mathbb H)$ is the closure of the space
$C^\infty_0(\Omega,\cc\otimes\mathbb H)$ in
$H^m(\Omega,\cc\otimes\mathbb H)$ with respect to the norm $\|\cdot
\|_{H^m}$. This space can be characterized as the set of all functions
$u\in H^m$ such that $Tr(\partial^\bfb u)= 0$ for any multiindex $\bfb\in\mathbb N^3$ with $|\bfb|\leq m-1$ (here the trace operator $Tr(\cdot)$ is defined as in \cite{Evans} Chapter $5$).\
Now we give to the problem (\ref{ProbAA}) the weak formulation in order to apply the Lax-Milgram lemma in the space $H^m_0(\Omega, \cc\otimes\hh)$.
\begin{remark}\label{r1b}
When $\Omega$ is bounded, we can endow $H^m_0(\Omega,\cc\otimes\mathbb H)$ with the scalar product
$$
\langle u,v\rangle_{D^m}:=\sum_{l=1}^3\langle \partial_{x_l}^m u,\partial_{x_l}^m v\rangle_{L^2}.
$$
The $H^m$-norm is equivalent to the norm
$$
\|u\|^2_{D^m}:=\langle u,u\rangle_{D^m}=\sum_{l=1}^3\|\partial_{x_l}^m u\|^2_{L^2}.
$$
This is a consequence of the following estimates
 $$ \|u\|_{D^m}\leq \|u\|_{H^m}\leq K(m) K_\Omega
 \sum_{|\bfb|=m}\|\partial^\bfb u  \|_{L^2}\leq K(m)K K_\Omega \|
 u\|_{D^m} ,
 $$
where the second inequality is obtained by Poincar\'e's inequality
applied repeatedly to the term $\|\partial^\bfb u\|_{L^2}$ for
$|\bfb|<m$ and 
$$
K_\Omega:=\sup(C_\Omega, C_\Omega^m). 
$$
The constant $K(m)$ represents the maximum number of times that
$\|\partial^\bfb u\|$ for some $|\bfb|=m$ appear in the left hand side
of the second inequality after the use of the Poincar\'e's
inequality.  The last estimate follows using the Fourier transform on the
terms $\partial^\bfb u$ when $|\bfb|=m$ and since there exists a
positive constant $K>0$ such that 
$$
\sum_{|\bfb|=m}|\xi^{2\bfb}|\leq K\sum_{l=1}^3 |\xi_l|^{2m} 
$$
(we can use the Fourier transform since any $u\in H^m_0(\Omega,\cc\otimes\hh)$ can be extended by $0$ outside $\Omega$ preserving the $H^m$-regularity in $\rr^3$). We define the constant $K(m,\Omega):=KK(m)K_\Omega$ (we will use several times this constant in the sequel and according to our necessity it could be rescaled by other constants wich depends on $m$). When $\Omega$ is unbounded $\|\cdot\|_{D^m}$ is not a norm, still we will use several times the estimate
$$ \sum_{|\beta|=m}\|\partial^\bfb u\|\leq K \| u\|_{D^m} .
$$
\end{remark}

\begin{remark}\label{r1}
We will use several times a classical argument of integration by parts that we describe now. Let $f_1,\, f_2\in H^k_0(\Omega,\rr)$ and $\partial^\bfb f_3,\, \partial^\bfb f_4\in L^2(\Omega)\cap L^\infty(\Omega)$ for any $\bfb\in\mathbb N^3$ with $|\bfb|\leq k$ then, integrating by parts $k$-times and recalling that the traces at the boundary of $\partial^\bfb f_2$ for all $|\bfb|<k$ are zero, we have that for any $i=1,\, 2,\, 3$
\[
\begin{split}
&\int_\Omega \partial_{x_i}^k(f_1) f_2f_3f_4\, dx =(-1)^k\int_{\Omega} f_1\partial_{x_i}^k(f_2f_3f_4)\, dx\\
&=(-1)^k\int_{\Omega} f_1\partial_{x_i}^k(f_2) f_3f_4\,
dx+(-1)^k\sum_{|\bft|=k\,\wedge\, t_1\leq
  k-1}\begin{pmatrix}k\\\bft\end{pmatrix} \int_{\Omega} f_1
\partial_{x_i}^{t_1} (f_2)\partial_{x_i}^{t_2}
(f_3)\partial_{x_i}^{t_3} (f_4)\, dx 
\\
&
=(-1)^k\sum_{|\bft|=k} \begin{pmatrix}k\\ \bft\end{pmatrix}
\int_{\Omega} f_1  \partial_{x_i}^{t_1} (f_2)\partial_{x_i}^{t_2}
(f_3)\partial_{x_i}^{t_3} (f_4)\, dx 
, 
\end{split}
\]
where $\bft=(t_1,t_2,t_3)\in\mathbb N^3$ and $\begin{pmatrix} k\\ \bft\end{pmatrix}=\frac{k!}{t_1!t_2!t_3!}$.
\end{remark}
From Definition \ref{TCOM},  we have 
\[
\begin{split}
\Q_{s}(T) & = T^2- 2s_0T + |s|^2\id
\\
&
 =   (-1)^{m}\left[\sum_{l=1}^3 a_l^2(x)\partial_{x_l}^{2m}+\sum_{l=1}^3\sum_{k=1}^m \begin{pmatrix}m\\k\end{pmatrix}a_l(x)\partial_{x_l}^k(a_l(x))\partial_{x_l}^{2m-k}\right.
 \\
 &\left.\  +\sum_{l<j}e_le_j \left(\sum_{k=1}^m \begin{pmatrix}m\\k\end{pmatrix}\left(a_l(x)\partial_{x_l}^k(a_j(x))\partial_{x_l}^{m-k}\partial_{x_j}^{m}-a_j(x)\partial_{x_j}^k(a_l(x))\partial_{x_j}^{m-k}\partial_{x_l}^{m}\right)\right)\right] - 2s_0T+ |s|^2\mathcal{I},
\end{split}
\]
where $\begin{pmatrix}m\\k\end{pmatrix}=\frac{m!}{k!(m-k)!}$, and the scalar part of $\Q_s(T)$ is
$$
{\rm Scal}(\Q_{s}(T)):=(-1)^m \left( \sum_{l=1}^3
  a_l^2(x)\partial_{x_l}^{2m}+\sum_{l=1}^3\sum_{k=1}^m \begin{pmatrix}m\\k\end{pmatrix}
  a_l(x)\partial_{x_l}^k(a_l(x))\partial_{x_l}^{2m-k}\right)
+|s|^2\mathcal{I} ,
$$
while the the vectorial part is 
\[
{\rm Vect}(\Q_{s}(T)):= \sum_{l<j}e_le_j
\left(\sum_{k=1}^m\begin{pmatrix}m\\k\end{pmatrix}\left(a_l(x)\partial_{x_l}^k(a_j(x))\partial_{x_l}^{m-k}\partial_{x_j}^{m}-a_j(x)\partial_{x_j}^k(a_l(x))\partial_{x_j}^{m-k}\partial_{x_l}^{m}\right)\right)
- 2s_0T .
\]

 We consider the bilinear form
\[
 \langle \Q_{s}(T)u,v\rangle_{L^2} = \int_{\Omega} \langle\Q_{s}(T)u(\vx), v(\vx)\rangle\,d\vx
\]
for functions $u,\, v\in C^{2m}_0(\overline{\Omega},\cc\otimes\mathbb{H})$. Note that
$$
\langle \mathcal Q_s(T)u,v \rangle_{L^2}= \langle
\operatorname{Scal}(\mathcal Q_s(T))u,v \rangle_{L^2}+\langle
\operatorname{Vect}(\mathcal Q_s(T))u,v \rangle_{L^2} .
$$
Using Remark \ref{r1} we have that 
\[
\begin{split}
&\langle \operatorname{Scal}(\mathcal Q_s(T))u,v \rangle_{L^2}
\\
&=(-1)^m\left(\sum_{l=1}^3\int_{\Omega} \langle a_l^2(x)\partial_{x_l}^{2m}( u), v\rangle \, dx+ \sum_{l=1}^3\sum_{k=1}^m \begin{pmatrix}m\\k\end{pmatrix}\int_{\Omega} \langle a_l(x)\partial_{x_l}^k(a_l(x))\partial_{x_l}^{2m-k}(u), v\rangle \, dx\right)\\
&\ +|s|^2\int_{\Omega} \overline u v\, dx=\sum_{l=1}^3\int_{\Omega} \langle a_l^2(x)\partial_{x_l}^{m}(u), \partial_{x_l}^{m}(v)\rangle \, dx+|s|^2\int_{\Omega} \langle u, v\rangle\, dx\\
&\ +\sum_{l=1}^3 \sum_{|\bft'|=m\,\wedge\, t_2'\leq m-1} \begin{pmatrix}m\\\bft'\end{pmatrix}\int_{\Omega} \langle\partial_{x_l}^{t_1'}(a_l^2(x))\partial_{x_l}^{m}(u), \partial_{x_l}^{t_2'}(v)\rangle\, dx \\
&\ +\sum_{l=1}^3\sum_{k=1}^m(-1)^k \sum_{|{\bft}|=m-k} \begin{pmatrix}m\\k\end{pmatrix} \begin{pmatrix}m-k\\\bft\end{pmatrix}\int_{\Omega} \langle \partial_{x_l}^{t_1}(a_l(x))\partial_{x_l}^{t_2+k}(a_l(x))\partial_{x_l}^{m}(u), \partial_{x_l}^{t_3}(v)\rangle \, dx,
\end{split}
\]
where $\bft'=(t_1',t_2')\in\mathbb N^2$ and also
\[
\begin{split}
&\langle \operatorname{Vect}(\mathcal Q_s(T))u,v \rangle_{L^2} \\
&=(-1)^m\sum_{l<j} \sum_{k=1}^m \begin{pmatrix}m\\k\end{pmatrix}\int_{\Omega} \langle e_le_j \left(a_l(x)\partial_{x_l}^k(a_j(x))\partial_{x_l}^{m-k}\partial_{x_j}^{m}(u)-a_j(x)\partial_{x_j}^k(a_l(x))\partial_{x_j}^{m-k}\partial_{x_l}^{m}(u)\right), v\rangle \, dx
\\
& \
-2s_0 \langle T(u), v \rangle_{L^2}\\
&= \sum_{l<j} \sum_{k=1}^m (-1)^{k} \sum_{|\bft|=m-k} \begin{pmatrix}m\\k\end{pmatrix}\begin{pmatrix}m-k\\\bft\end{pmatrix}\left(\int_{\Omega} \langle e_le_j \partial_{x_l}^{t_1}(a_l(x))\partial_{x_l}^{t_2+k}(a_j(x))\partial_{x_j}^{m}(u), \partial_{x_l}^{t_3}v\rangle \, dx\right.
\\
&\ \left. -\int_{\Omega} \langle e_le_j\partial_{x_j}^{t_1}(a_j(x))\partial_{x_j}^{t_2+k}(a_l(x))\partial_{x_l}^{m}(u),\partial_{x_j}^{t_3}v\rangle\, dx -2s_0 \langle T(u), v \rangle_{L^2}\right).
\end{split}
\]
Relying on the above considerations we can give the following two definitions.

\begin{definition}\label{weakformbis}
Let $\Omega$ be  a bounded domain (or an unbounded domain) in $\mathbb R^3$ with the boundary $\partial\Omega$ of class $C^1$, let $a_1$, $a_2$, $a_3\in C^m(\overline{\Omega}, \mathbb{R})$ (or in the case of the unbounded domains $a_1$, $a_2$, $a_3\in C^m(\overline{\Omega}, \mathbb{R})\cap L^\infty(\Omega)$ such that $\partial^\bfb a_j\in L^\infty(\Omega)$ for any $j=1,\, 2,\, 3$ and for any $\bfb\in\mathbb N^3$ such that $|\bfb|\leq m$).
We define the bilinear form:
\begin{equation}\label{b1bis}
\begin{split}
 b_s(u,v):& =\sum_{l=1}^3\int_{\Omega} \langle a_l^2(x)\partial_{x_l}^{m}(u), \partial_{x_l}^{m}(v)\rangle \, dx+|s|^2\int_{\Omega} \langle u, v\rangle\, dx\\
&\ +\sum_{l=1}^3 \sum_{|\bft'|=m\,\wedge\, t_2\leq m-1}\begin{pmatrix}m\\\bft'\end{pmatrix}\int_{\Omega} \langle\partial_{x_l}^{t_1}(a_l^2(x))\partial_{x_l}^{m}(u), \partial_{x_l}^{t_2}(v)\rangle\, dx \\
&\ +\sum_{l=1}^3\sum_{k=1}^m (-1)^k \sum_{|\bft|=m-k} \begin{pmatrix}m\\k\end{pmatrix}\begin{pmatrix}m-k\\\bft\end{pmatrix}\int_{\Omega} \langle \partial_{x_l}^{t_1}(a_l(x))\partial_{x_l}^{t_2+k}(a_l(x))\partial_{x_l}^{m}(u), \partial_{x_l}^{t_3}(v)\rangle \, dx
\\
&\ + \sum_{l<j} \sum_{k=1}^m (-1)^{k} \sum_{|\bft|=m-k} \begin{pmatrix}m\\k\end{pmatrix}\begin{pmatrix}m-k\\\bft\end{pmatrix}\left(\int_{\Omega} \langle e_le_j \partial_{x_l}^{t_1}(a_l(x))\partial_{x_l}^{t_2+k}(a_j(x))\partial_{x_j}^{m}(u), \partial_{x_l}^{t_3}v\rangle \, dx\right.
\\
&\ \left. -\int_{\Omega} \langle e_le_j
  \partial_{x_j}^{t_1}(a_j(x))\partial_{x_j}^{t_2+k}(a_l(x))\partial_{x_l}^{m}(u),\partial_{x_j}^{t_3}v\rangle\,
  dx -2s_0 \langle T(u), v \rangle_{L^2}\right) ,
\end{split}
\end{equation}
 for all functions $u,v \in H_0^m(\Omega,\cc\otimes\hh)$.
\end{definition}

\begin{definition}\label{wf}
Let $\Omega$ be a bounded domain in $\mathbb{R}^3$ with $C^1$ boundary (or an unbounded domain which satisfy property $(R)$).
We say that $u\in H^m_0$ is the weak solution of the existence problem in \ref{ProbAA} for some $s\in \mathbb{H}$ and a given $F\in L^2(\Omega, \cc\otimes\hh)$, if we have
$$
b_s(u,v)=\langle F,v\rangle_{L^2},\ \ \ \ {\rm for \ all} \ \ v\in H^m_0,
$$
 where $b_s$ is the bilinear form defined in Definition \ref{weakformbis}.
\end{definition}

\section{\bf Weak solution of the Problem \ref{ProbAA} when $\Omega$ is bounded}\label{s3b}
To prove existence and uniqueness of the weak solutions of the Problem \ref{ProbAA} in the case $\Omega$ is bounded, it will be sufficient to show that the bilinear forms $b_s(\cdot,\cdot)$, in Definition \ref{weakformbis},
 are continuous  and coercive in $H^m_0(\Omega,\cc\otimes\hh)$.

  First we prove the continuity in Section \ref{cont_b}. The coercivity will be proved in Section \ref{coer2_b} when $s=js_1$ for $j\in\mathbb S$ and $s_1\in\rr$ with $s_1\neq 0$. As a direct consequence of the coercivity, we will prove an $L^2$ estimate for the weak solution $u$ that belongs to $H^m_0(\Omega,\cc\otimes\hh)$ and also we will prove an $L^2$ estimate
    for the term $T(u)$.

\subsection{The continuity of the bilinear form $b_s(\cdot,\cdot)$}\label{cont_b}
The bilinear form
$$
b_s(\cdot, \cdot): H^m_0(\Omega,\cc\otimes\hh)\times H^m_0(\Omega,\cc\otimes\hh)\to \mathbb{H},
$$
for some $s\in \mathbb{H}$, is continuous if there exists a positive constant $C(s)$ such that
$$
|b_s(u,v)|\leq C(s) \|u\|_{D^m}\|v\|_{D^m},\ \ \ \ {\rm for \ all} \ \ u,v\in H^m(\Omega,\cc\otimes\hh).
$$
 We note that the constant $C(s)$ depends on $s\in \mathbb{H}$ but does not depend on $u$ and $v\in H^m(\Omega,\cc\otimes\hh)$.

\begin{proposition}[Continuity of $b_s$]\label{p1_b}
Let $\Omega$ be a  bounded domain in $\mathbb R^3$ with boundary $\partial\Omega$ of class $C^1$. Assume that $a_1$, $a_2$, $a_3\in C^m(\overline{\Omega}, \mathbb{R})$ then we have
\begin{equation}\label{cont_est_b_1}
\left| \sum_{l=1}^3\int_{\Omega} \langle a_l^2(x)\partial_{x_l}^{m}(u), \partial_{x_l}^{m}(v)\rangle \, dx \right|\leq C_{1,\, m,\, a_j,\, \Omega} \|u\|_{D^m}\|v\|_{D^m}
\end{equation}
\begin{equation}\label{cont_est_b_2}
\Bigg|\sum_{l=1}^3 \sum_{|\bft'|=m\,\wedge\, t_2\leq m-1}\begin{pmatrix}m\\\bft'\end{pmatrix}\int_{\Omega} \langle\partial_{x_l}^{t_1'}(a_l^2(x))\partial_{x_l}^{m}(u), \partial_{x_l}^{t_2'}(v)\rangle\, dx\Bigg|\leq C_{2,\, m,\, a_j,\, \Omega}\|u\|_{D^m}\|v\|_{D^m}
\end{equation}
\begin{equation}\label{cont_est_b_3}
\begin{split}
&\Bigg|\sum_{l=1}^3\sum_{k=1}^m (-1)^k \sum_{|\bft|=m-k} \begin{pmatrix}m\\k\end{pmatrix}\begin{pmatrix}m-k\\\bft\end{pmatrix}\int_{\Omega} \langle \partial_{x_l}^{t_1}(a_l(x))\partial_{x_l}^{t_2+k}(a_l(x))\partial_{x_l}^{m}(u), \partial_{x_l}^{t_3}(v)\rangle \, dx\Bigg|\\
&\leq C_{3,\, m,\, a_j,\, \Omega} \|u\|_{D^m}\|v\|_{D^m}
\end{split}
\end{equation}
\begin{equation}\label{cont_est_b_4}
\begin{split}
&\Bigg|\sum_{l<j} \sum_{k=1}^m (-1)^{k} \sum_{|\bft|=m-k} \begin{pmatrix}m\\k\end{pmatrix}\begin{pmatrix}m-k\\\bft\end{pmatrix}\left(\int_{\Omega} \langle e_le_j \partial_{x_l}^{t_1}(a_l(x))\partial_{x_l}^{t_2+k}(a_j(x))\partial_{x_j}^{m}(u), \partial_{x_l}^{t_3}v\rangle \, dx\right.
\\
&\left.
-\int_{\Omega} \langle e_le_j \partial_{x_j}^{t_1}(a_j(x))
\partial_{x_j}^{t_2+k}(a_l(x))\partial_{x_l}^{m}(u),\partial_{x_j}^{t_3}v\rangle\,
dx\right) -2s_0 \langle T(u), v \rangle_{L^2}\Bigg|\leq C_{4,\, m,\,
a_j,\, \Omega,\,s}\|u\|_{D^m}\|v\|_{D^m} .
\end{split}
\end{equation}
The previous constants can be estimated as follows
$$
C_{1,\, m,\, a_j,\, \Omega}\leq \max_{l=1,2,3}\sup_{ x\in\Omega}(a_l^2(x)), \quad C_{2,\, m,\, a_j,\, \Omega}\leq CK(m,\Omega) \max_{\substack{l=1,2,3,\\ t=1,\dots, m}}\sup_{ x\in\Omega} (|\partial_{x_l}^t (a_l^2(x))|)$$
$$ C_{3,\, m,\, a_j,\, \Omega}\leq CK(m,\Omega) \left(\max_{\substack{l=1,2,3,\\ t=0,\dots, m}}\sup_{ x\in\Omega} (|\partial_{x_l}^t (a_l(x))|)\right)^2 $$
$$C_{4,\, m,\, a_j,\, \Omega,\, s}\leq CK(m,\Omega)
\left[\left(\max_{\substack{l=1,2,3,\\ t=0,\dots, m}}\sup_{
      x\in\Omega} (|\partial_{x_l}^t
    (a_j(x))|)\right)^2+|s_0|\max_{l=1,2,3}\sup_{
    x\in\Omega}(|a_l(x)|)\right] ,
$$ 
where the constant $C$ is the sum of the maximum  of the integrals
in the inequalities \eqref{cont_est_b_2},  \eqref{cont_est_b_3} and
\eqref{cont_est_b_4}.

Moreover, for any $s\in\hh$ there exists a
positive constant $C(s)$ such that for any  $(u,v)\in H^m_0(\Omega,
\cc\otimes\mathbb H)\times H^m_0(\Omega, \cc\otimes\mathbb H)$ the
bilinear form $b_s(\cdot,\cdot)$ satisfy the estimate 
\begin{equation}\label{bilcont}
|b_s(u,v)|\leq C(s) \|u\|_{D^m}\|v\|_{D^m},
\end{equation}
that is, $b_s(\cdot,\cdot)$ is a continuous bilinear form.
\end{proposition}
\begin{proof}
The inequality \eqref{cont_est_b_1} is a direct consequence of the
boundedness of the coefficients $a_l$'s and the H\"older
inequality. The estimates  \eqref{cont_est_b_2}, \eqref{cont_est_b_3}
and \eqref{cont_est_b_4} can be proved in a similar way and so we
briefly explain how to prove the inequality
\eqref{cont_est_b_2}. First,  to each integral that is on
the left hand side of the equation, we repeatedly apply  H\"older's and
Poincar\'e's inequalities.  In this way each integral can be estimated by a
constant depending on the sup norm of $\partial_{x_l}^s a_t$'s 
times $\|u\|_{D^m}\|v\|_{D^m}$. Second, we sum up term by
term. Finally, the continuity of $b_s(\cdot,\cdot)$ is a direct
consequence of the previous estimates. 
\end{proof}

\subsection{Weak solution of the Problem
\ref{ProbAA}\label{coer2_b}}

\begin{theorem}\label{t3_b}
Let $\Omega$ be a bounded domain in $\mathbb R^3$ with boundary $\partial\Omega$ of class $C^1$. Let $T$ be the operator defined in (\ref{TCOM}) with coefficients $a_1$, $a_2$, $a_3\in C^m(\overline{\Omega}, \mathbb{R})$ and set
$$
M:=K(m,\Omega)^2\biggl(\sup_{ \substack{t=1,\dots,m\\ l=1,2,3\\ x\in\Omega}} |\partial_{x_l}^t (a_l(x))|\biggr)^2.
$$
Suppose
\begin{equation}\label{kappaomeg_b}
 C_T:=\min_{l=1,2,3}\inf_{x\in\Omega} (a^2_l(x))>0 \ \ \ \
  \frac{C_T}{2}- M>0.
\end{equation}
then:

(I) The boundary value Problem (\ref{ProbAA})
has a unique weak solution $u\in H^m_0(\Omega,\hh)$, for $s\in\hh\setminus \{0\}$ with $\Re(s)=0$,
and
\begin{equation}\label{l2es1bis_b}
 \|u\|^2_{L^2}\leq \frac 1{s^2} \Re(b_s(u,u)).
\end{equation}

(II) The following estimate holds
\begin{equation}\label{l2es2_b}
 \|T(u)\|_{L^2}^2\leq C_1\Re(b_s(u,u)),
 \end{equation}
for every $u\in H^1_0(\Omega,\hh)$, and $s\in\hh\setminus \{0\}$ with $\Re(s)=0$,
where
$$
C_1:=\frac{C_T-2M}{2C_T}.
$$
\end{theorem}

Observe that we assume that the minimum of $\sum_{l=1}^3 a_l^2(x)$
for $x\in\overline{\Omega}$ is strictly positive and in fact greater
than $2M$. This fact can always be achieved by modify $T$ by adding a suitable
constant to just  one of the coefficients $a_l$, $l=1,2,3$. 

\begin{proof} In order to use the Lax-Milgram Lemma to prove the existence and the uniqueness of the solution for the
weak formulation of the problem, it is sufficient to prove the coercivity of the bilinear form $b_s(\cdot,\cdot)$
   in Definition \ref{b1bis}
since the continuity is proved in Proposition \ref{p1_b}. First we write explicitly ${\rm Re}\, b_{js_1}(u,u)$,
 where we have set $s=js_1$, for $s_1\in \mathbb{R}$ and  $j\in \mathbb{S}$:
	
\begin{equation}\label{e1bis}
\begin{split}
\Re\,  b_{js_1} & (u,u) =
s_1^2\|u\|^2_{L^2}+ \sum_{\ell=1}^3\| a_\ell\partial_{x_\ell}^m u\|^2_{L^2}
\\
&+\Re\Bigg(\sum_{l=1}^3 \sum_{|\bft'|=m\,\wedge\, t_2\leq m-1}\begin{pmatrix}m\\\bft'\end{pmatrix}\int_{\Omega} \langle\partial_{x_l}^{t_1'}(a_l^2(x))\partial_{x_l}^{m}(u), \partial_{x_l}^{t_2'}(u)\rangle\, dx \\
&+\sum_{l=1}^3\sum_{k=1}^m (-1)^k \sum_{|\bft|=m-k} \begin{pmatrix}m\\k\end{pmatrix}\begin{pmatrix}m-k\\\bft\end{pmatrix}\int_{\Omega} \langle \partial_{x_l}^{t_1}(a_l(x))\partial_{x_l}^{t_2+k}(a_l(x))\partial_{x_l}^{m}(u), \partial_{x_l}^{t_3}(u)\rangle \, dx
\\
&+ \sum_{l<j} \sum_{k=1}^m (-1)^{k} \sum_{|\bft|=m-k} \begin{pmatrix}m\\k\end{pmatrix}\begin{pmatrix}m-k\\\bft\end{pmatrix}\left( \int_{\Omega} \langle e_le_j \partial_{x_l}^{t_1}(a_l(x))\partial_{x_l}^{t_2+k}(a_j(x))\partial_{x_j}^{m}(u), \partial_{x_l}^{t_3}u\rangle \, dx\right.
\\
&\left. -\int_{\Omega} \langle e_le_j \partial_{x_j}^{t_1}(a_j(x))\partial_{x_j}^{t_2+k}(a_l(x))\partial_{x_l}^{m}(u),\partial_{x_j}^{t_3}u\rangle\, dx\right)\Bigg).
\end{split}
\end{equation}
We see that the first two terms are positive. What remain to estimate
are the three summations of the integrals in \eqref{e1bis}. Since they
can be treated in the same way, we explain in detail the estimate for
the second. By H\"older's inequality and the repeatedly use of
Poincar\'e's inequality,  we have for $t_1=0$ 
\[
\begin{split}
&\left|\int_{\Omega} \langle a_l(x)\partial_{x_l}^{t_2+k}(a_l(x))\partial_{x_l}^{m}(u), \partial_{x_l}^{t_3}(v)\rangle \, dx\right| \leq \epsilon \|a_l\partial_{x_l}^m u\|^2+\frac 1{\epsilon}\|\partial_{x_l}^{t_2+k}(a_l)\partial_{x_l}^{t_3}(u)\|^2\\
&\leq \epsilon \|a_l\partial_{x_l}^m u\|^2+\frac 1{\epsilon}\left(\sup_{x\in\Omega}(\partial_{x_l}^{t_2+k}(a_l))\right)^2\|\partial_{x_l}^{t_3}(u)\|^2\\
&\leq \epsilon \|a_l\partial_{x_l}^m u\|^2+\frac 1{\epsilon}\left(\sup_{x\in\Omega}(\partial_{x_l}^{t_2+k}(a_l))\left(C_\Omega\right)^{m-t_3}\right)^2\sum_{|\beta|=m}\| \partial^{\beta}(u)\|^2\\
&\leq \epsilon \|a_l\partial_{x_l}^m u\|^2+\frac
1{\epsilon}\left(\sup_{x\in\Omega}(\partial_{x_l}^{t_2+k}(a_l))KK(m)\left(C_\Omega\right)^{m-t_3}\right)^2\|
u\|^2_{D^m} .
\end{split}
\]
In the case $t_1\neq 0$, since $k>0$, we have
\[
\begin{split}
&\left|\int_{\Omega} \langle \partial_{x_l}^{t_1}(a_l(x))\partial_{x_l}^{t_2+k}(a_l(x))\partial_{x_l}^{m}(u), \partial_{x_l}^{t_3}(v)\rangle \, dx \right| \leq \frac 12 \|\partial_{x_l}^{t_1} (a_l)\partial_{x_l}^m u\|^2+\frac 12\|\partial_{x_l}^{t_2+k}(a_l)\partial_{x_l}^{t_3}(u)\|^2\\
&\leq \frac 12 \sup_{x\in\Omega}\left(\left| \partial_{x_l}^{t_1} a_l
  \right|^2\right)\|\partial_{x_l}^m u\|^2+ \frac 12
\sup_{x\in\Omega}\left| \partial_{x_l}^{t_2+k} a_l \right|^2 \left(
  KK(m)\left(C_\Omega\right)^{m-t_3}\right)^2\| u\|^2_{D^m} .
\end{split}.
\]
Summing up all the previous estimates of the terms in the second summation and rescaling the constants $K(m,\Omega)$ multiplying it by a constant which depends only on $m$, we obtain
\[
\begin{split}
&\left| \sum_{l=1}^3 \sum_{|\bft'|=m\,\wedge\, t_2\leq m-1}\sum_{k=0}^{t_1}\begin{pmatrix}m\\\bft'\end{pmatrix}\begin{pmatrix}t_1\\k\end{pmatrix}\int_{\Omega} \langle\partial_{x_l}^{k}(a_l(x))\partial_{x_l}^{t_1-k}(a_l(x))\partial_{x_l}^{m}(u), \partial_{x_l}^{t_2}(u)\rangle\, dx  \right|\\
&\leq \epsilon\sum_{l=1}^3\|a_l\partial_{x_l}^m u\|^2 + \left(\frac{K(m,\Omega)^2}{\epsilon}\max_{\substack{j,\,l=1,2,3\\ t=1,\dots,\, m}}\,\sup_{x\in\Omega}\left(\left| \partial_{x_j}^t a_l \right|^2\right)\right) \|u\|^2_{D^m}.
\end{split}
\]
Analogous estimates also holds for the other summation of integrals. Eventually, summing up all the estimates, choosing $\epsilon$ in such a way that $\sum_{l=1}^3\|a_l\partial_{x_l}^m u\|^2-C\epsilon\sum_{l=1}^3\|a_l\partial_{x_l}^m u\|^2\geq \frac 12 \sum_{l=1}^3\|a_l\partial_{x_l}^m u\|^2$ (where $C$ depends only on $m$) and uniformazing the constants $K(m,\Omega)$,  we obtain
\[
\begin{split}
\Re\, b_{js_1}  (u,u)  & \geq  s_1^2 \|u\|^2_{L^2}
 + \left(\frac 12 C_T-M\right)\| u\|^2_{D^m}.
\end{split}
\]
and, moreover,
\begin{equation}\nonumber
 \frac 12 C_T-M>0.
\end{equation}
Thus the quadratic form $b_{js_1}(\cdot,\cdot)$ is coercive for every $s_1\in \mathbb{R}$.
In particular we have
\begin{equation}\label{e5_b}
\Re\, b_{js_1}(u,u)\geq s_1^2\|u\|^2_{L^2}\quad\textrm{and}\quad  \Re\, b_{js_1}(u,u)\geq\left(\frac 12 C_T-M\right)\| u\|^2_{D^m}.
\end{equation}
By the Lax-Milgram Lemma, we have that for any $w\in L^2(\Omega, \cc\otimes\mathbb H)$ there exists $u_w\in  H^m_0(\Omega,\hh)$,
for $s_1\in \mathbb{R}\setminus \{0\}$ and  $j\in \mathbb{S}$, such that
$$ b_{js_1}(u_w,v)= \langle w, v \rangle_{L^2},\ \ \ {\rm for \ all}\ \ v\in  H^m_0(\Omega,\hh). $$
What remains to prove is the inequality \eqref{l2es2_b}. Applying the second of the inequalities in \eqref{e5_b} and observing that
$$ \sum_{\ell=1}^3\|\partial_{x_\ell}u\|_{L^2}^2\leq \frac{1}{C_T}\sum_{\ell=1}^3\|a_\ell\partial_{x_\ell}u\|_{L^2}^2 $$
we have:
\[
\begin{split}
\Re\,  b_{js_1}  (u,u) &\geq\frac 12 \sum_{l=1}^3\|a_l \partial_{x_l}^m u\|-M\| u\|^2_{D^m}\\
&
\geq  \frac 12\sum_{\ell=1}^3\| a_\ell\partial_{x_\ell}^mu\|^2_{L^2}-\frac{M}{C_T}\sum_{\ell=1}^3\| a_\ell\partial_{x_\ell}^m u\|^2_{L^2}
\\
&
\geq \frac {C_T-2M}{2C_T}\sum_{\ell=1}^3\| a_\ell\partial_{x_\ell}^m u\|^2_{L^2}\geq C_1\|Tu\|_{L^2}^2,
\end{split}
\]
where we have set
$$
C_1:=\frac{C_T-2M}{2C_T}
$$
and this concludes the proof.
\end{proof}

\section{\bf Weak solution of the Problem \ref{ProbAA} for unbounded $\Omega$}\label{s3unb}

To prove existence and uniqueness of the weak solutions of the Problem \ref{ProbAA} in the case of $\Omega$ unbounded, we proceed as in the previous section.

  First we prove the continuity in Section \ref{cont_unb}. The coercivity will be proved in Section \ref{coer2} when $s=js_1$ for $j\in\mathbb S$ and $s_1\in\rr$ with $s_1\neq 0$. As a direct consequence of the coercivity, we will prove an $L^2$ estimate for the weak solution $u$ that belongs to $H^m_0(\Omega,\cc\otimes\hh)$ and also we will prove an $L^2$ estimate
    for the term $T(u)$.

\subsection{The continuity of the bilinear form $b_s(\cdot,\cdot)$}\label{cont_unb}
The bilinear form
$$
b_s(\cdot, \cdot): H^m_0(\Omega,\cc\otimes\hh)\times H^m_0(\Omega,\cc\otimes\hh)\to \mathbb{H},
$$
for some $s\in \mathbb{H}$, is continuous if there exists a positive constant $C(s)$ such that
$$
|b_s(u,v)|\leq C(s) \|u\|_{H^m}\|v\|_{H^m},\ \ \ \ {\rm for \ all} \ \ u,v\in H^m(\Omega,\cc\otimes\hh).
$$
 We note that the constant $C(s)$ depends on $s\in \mathbb{H}$ but does not depend on $u$ and $v\in H^m(\Omega,\cc\otimes\hh)$.

\begin{proposition}[Continuity of $b_s$]\label{p1}
Let $\Omega$ be an unbounded domain in $\mathbb R^3$ with boundary $\partial\Omega$ of class $C^1$. Assume that $a_1$, $a_2$, $a_3\in C^m(\overline{\Omega}, \mathbb{R})\cap L^\infty(\Omega)$ and that $\partial^\beta a_j\in L^\infty(\Omega)$ for any $j=1,\, 2,\, 3$ and for any $\beta\in\mathbb N^3$, $|\beta|\leq m$. Then we have
\begin{equation}\label{cont_est_unb_1}
\left| \sum_{l=1}^3\int_{\Omega} \langle a_l^2(x)\partial_{x_l}^{m}(u), \partial_{x_l}^{m}(v)\rangle \, dx \right|\leq C_{1,\, m,\, a_j,\, \Omega} \|u\|_{H^m}\|v\|_{H^m}
\end{equation}
\begin{equation}\label{cont_est_unb_2}
\Bigg|\sum_{l=1}^3 \sum_{|\bft'|=m\,\wedge\, t_2\leq m-1}\begin{pmatrix}m\\\bft'\end{pmatrix}\int_{\Omega} \langle\partial_{x_l}^{t_1'}(a_l^2(x))\partial_{x_l}^{m}(u), \partial_{x_l}^{t_2'}(v)\rangle\, dx\Bigg|\leq C_{2,\, m,\, a_j,\, \Omega}\|u\|_{H^m}\|v\|_{H^m}
\end{equation}
\begin{equation}\label{cont_est_unb_3}
\begin{split}
\Bigg|\sum_{l=1}^3 \sum_{k=1}^m (-1)^{k} \sum_{|\bft|=m-k} \begin{pmatrix}m\\k\end{pmatrix}\begin{pmatrix}m-k\\\bft\end{pmatrix}& \int_{\Omega} \langle \partial_{x_l}^{t_1}(a_l(x))\partial_{x_l}^{t_2+k}(a_l(x))\partial_{x_l}^{m}(u), \partial_{x_l}^{t_3}(v)\rangle \, dx\Bigg|\\
&\leq C_{3,\, m,\, a_j,\, \Omega} \|u\|_{H^m}\|v\|_{H^m}
\end{split}
\end{equation}
\begin{equation}\label{cont_est_unb_4}
\begin{split}
&\Bigg|\sum_{l<j} \sum_{k=1}^m (-1)^{k} \sum_{|\bft|=m-k} \begin{pmatrix}m\\k\end{pmatrix}\begin{pmatrix}m-k\\\bft\end{pmatrix} \int_{\Omega} \langle e_le_j \partial_{x_l}^{t_1}(a_l(x))\partial_{x_l}^{t_2+k}(a_j(x))\partial_{x_j}^{m}(u), \partial_{x_l}^{t_3}v\rangle \, dx
\\
&
-\int_{\Omega} \langle e_le_j \partial_{x_j}^{t_1}(a_j(x))
\partial_{x_j}^{t_2+k}(a_l(x))\partial_{x_l}^{m}(u),\partial_{x_j}^{t_3}v\rangle\, dx -2s_0 \langle T(u), v \rangle_{L^2}\Bigg|\leq C_{4,\, m,\, a_j,\, \Omega,\,s}\|u\|_{H^m}\|v\|_{H^m}
\end{split}
\end{equation}
where the previous constants can be estimated as follows
$$
C_{1,\, m,\, a_j,\, \Omega}\leq \max_{l=1,2,3}\sup_{ x\in\Omega}(a_l^2(x)), \quad C_{2,\, m,\, a_j,\, \Omega}\leq CK(m) \max_{\substack{l=1,2,3\\ t=1,\dots, m}}\sup_{ x\in\Omega} (|\partial_{x_l}^t (a_l^2(x))|)$$
$$ C_{3,\, m,\, a_j,\, \Omega}\leq CK(m) \left(\max_{\substack{l=1,2,3\\t=0,\dots, m}}\sup_{ x\in\Omega} (|\partial_{x_l}^t (a_l(x))|)\right)^2 $$
$$C_{4,\, m,\, a_j,\, \Omega,\, s}\leq CK(m) \left[\left(\max_{\substack{l=1,2,3\\ t=0,\dots, m}}\sup_{ x\in\Omega} (|\partial_{x_l}^t (a_j(x))|)\right)^2+|s_0|\max_{l=1,2,3}\sup_{ x\in\Omega}(|a_l(x)|)\right]
$$
 where the constant $C$ is the sum of the maximum  of the integrals 
in  the inequalities \eqref{cont_est_b_2},  \eqref{cont_est_b_3} and
\eqref{cont_est_b_4}.

Moreover, for any $s\in\hh$ there exists a
positive constant $C(s)$ such that for any  $(u,v)\in H^m_0(\Omega,
\cc\otimes\mathbb H)\times H^m_0(\Omega, \cc\otimes\mathbb H)$ the
bilinear form $b_s(\cdot,\cdot)$ satisfy the estimate 
\begin{equation}\label{bilcont}
|b_s(u,v)|\leq C(s) \|u\|_{H^m}\|v\|_{H^m},
\end{equation}
i.e., $b_s(\cdot,\cdot)$ is a continuous bilinear form.
\end{proposition}
\begin{proof}
The inequality \eqref{cont_est_unb_1} is a direct consequence of the boundedness of the coefficients $a_l$'s and the H\"older inequality. The estimates  \eqref{cont_est_unb_2}, \eqref{cont_est_unb_3} and \eqref{cont_est_unb_4} can be proved in a similar way and so we briefly explain how to prove the inequality \eqref{cont_est_unb_2}. First we apply the H\"older inequality to each integral that belongs to the left hand side. Thus each integral can be estimated by a constant (that depends on the sup norm of $\partial_{x_l}^s a_t$'s) times $\|\partial^m_{x_l} u\|\|\partial^\beta v\|$ for $|\beta|\leq m$ and $l=1,\, 2,\, 3$. Since  $\|\partial^m_{x_l} u\|\|\partial^\beta v\|\leq \|u\|_{H^m}\|v\|_{H^m}$ for any $|\beta|\leq m$ and $j=1,\,2,\, 3$, we sum up term by term to get the desired estimate. Eventually, the continuity of $b_s(\cdot,\cdot)$ is a direct consequence of the previous estimates.
\end{proof}

\subsection{Weak solution of the Problem \ref{ProbAA}}\label{coer2}
For the coercivity we need the following results (see \cite{EM80})
which is the corresponding of  Poincar\'e's inequality in the case of the unbounded domains.
\begin{theorem}\label{pun}
Let $\Omega$ be a domain in $\rr^3$ which satisfies the property $(R)$. Let $\phi\in C^1(\overline \Omega)$ be a bounded positive function for which there exists a positive constant $\lambda$ such that
\begin{equation}\label{wh1}
\frac{d}{dr}\left[r^{2}\phi(r\omega)\right]\leq -\lambda r^{2}\phi(r\omega)
\end{equation}
for all $\omega\in\mathbb S^{2}$. Then for each $u\in C^\infty_0(\Omega,\rr)$ and any $(1\leq p<\infty)$ we have
\begin{equation}\nonumber
\int_{\Omega}|u|^p\phi \, dx\leq
\left(\frac{p}{\lambda}\right)^p\int_\Omega |\nabla u|^p \phi\, dx .
\end{equation}
\end{theorem}
We call $\phi$ a weight function for $\Omega$ with rate of exponentially decaying $\lambda$. As explained in \cite{EM80}, we can observe that the graph of $\phi$ has the shape of an exponentially decaying hill. An example of a weight function for all the domains which satisfy the property $(R)$ is
$$ \phi(x)=\frac{e^{-\lambda|x-P|}}{|x-P|^2}\quad\textrm{when $P\notin\Omega$} $$
 If the domain is contained in a half-space, after an affine changing of variables, we can suppose that it is contained in $\{x\in\rr^3:\, x_1>0\}$. The previous theorem holds also requiring the weight function $\phi$ satisfies the assumption
 $$ \partial_{x_1}\phi(x)\leq -\lambda \phi(x) .
 $$
instead of the assumption \eqref{wh1}. In this case the graph of
$\phi$ has the shape of an exponentially decaying ridge and an example
of this kind of function is 
$$\phi=e^{-\lambda x_1}$$
\begin{theorem}\label{t3}
Let $\Omega$ be an unbounded domain in $\mathbb R^3$ such that it
satisfies the property $(R)$ and with its boundary $\partial\Omega$ of
class $C^1$. Let $\phi$ a bounded weight function for $\Omega$ with
rate of exponential decay $\lambda$. Let $T$ be the operator defined
in (\ref{TCOM}) with coefficients $a_1$, $a_2$, $a_3\in
C^m(\overline{\Omega}, \mathbb{R})\cap L^\infty(\Omega)$ such that
$|\partial_{x_l}^t a_j(x)|^2\leq C \phi$ for some positive constant
$C$ and for all $t=1,\, \dots, \, m$ and $j,\,l=1,\, 2,\, 3$. We set 
$$
M:=K(m)\max_{\substack{l=1,2,3\\t=1,\dots,\,
    m}}\,\sup_{x\in\Omega}\left(\left| \partial_{x_l}^t a_l
  \right|^2\right)+K(m,\phi,\lambda) ,
$$
where $K(m)$ and $K(m,\phi,\lambda)$ are positive constants that depends on the order $m$ of the operator, on $\phi$ and on $\lambda$. We suppose that
\begin{equation}\label{kappaomeg}
 C_T:=\min_{l=1,2,3}\inf_{x\in\Omega} (a^2_l(x))>0, \ \ \ \
  \frac 12 C_T- M>0.
\end{equation}
Then:

(I) The boundary value Problem (\ref{ProbAA})
has a unique weak solution $u\in H^m_0(\Omega,\hh\otimes \cc)$, for $s\in\hh\setminus \{0\}$ with $\Re(s)=0$,
and
\begin{equation}\label{l2es1bis}
 \|u\|^2_{L^2}\leq \frac 1{|s|^2} \Re(b_s(u,u)).
\end{equation}

(II)
Moreover, we have the following estimate
\begin{equation}\label{l2es2}
 \|T(u)\|_{L^2}^2\leq C_1\Re(b_s(u,u)),
 \end{equation}
for every $u\in  H^m_0(\Omega,\hh)$, and $s\in\hh\setminus \{0\}$ with $\Re(s)=0$,
where
$$
C_1:=\frac{C_T-2M}{2C_T}.
$$
\end{theorem}
\begin{proof} In order to use the Lax-Milgram Lemma to prove the existence and the uniqueness of the solution for the
weak formulation of the problem, it is sufficient to prove the coercivity of the bilinear form $b_s(\cdot,\cdot)$
   in Definition \ref{b1bis}
since the continuity is proved in Proposition \ref{p1}. First we write explicitly ${\rm Re}\, b_{{\bf j}s_1}(u,u)$,
 where we have set $s={\bf j}s_1$, for $s_1\in \mathbb{R}$ and  ${\bf j}\in \mathbb{S}$:
	
\begin{equation}\label{e1}
\begin{split}
\Re\,  b_{{\bf j}s_1} & (u,u) =
s_1^2\|u\|^2_{L^2}+ \sum_{\ell=1}^3\| a_\ell\partial_{x_\ell}^m u\|^2_{L^2}
\\
&+\Re\Bigg(\sum_{l=1}^3 \sum_{|\bft'|=m\,\wedge\, t_2\leq m-1}\sum_{k=0}^{t_1}\begin{pmatrix}m\\\bft'\end{pmatrix}\begin{pmatrix}t_1\\k\end{pmatrix}\int_{\Omega} \langle\partial_{x_l}^{k}(a_l(x))\partial_{x_l}^{t_1-k}(a_l(x))\partial_{x_l}^{m}(u), \partial_{x_l}^{t_2}(u)\rangle\, dx \\
&+\sum_{l=1}^3 \sum_{k=1}^m (-1)^{k} \sum_{|\bft|=m-k} \begin{pmatrix}m\\k\end{pmatrix}\begin{pmatrix}m-k\\\bft\end{pmatrix}\int_{\Omega} \langle \partial_{x_l}^{t_1}(a_l(x))\partial_{x_l}^{t_2+k}(a_l(x))\partial_{x_l}^{m}(u), \partial_{x_l}^{t_3}(u)\rangle \, dx
\\
&+ \sum_{l<j} \sum_{k=1}^m (-1)^{k} \sum_{|\bft|=m-k} \begin{pmatrix}m\\k\end{pmatrix}\begin{pmatrix}m-k\\\bft\end{pmatrix}\left( \int_{\Omega} \langle e_le_j \partial_{x_l}^{t_1}(a_l(x))\partial_{x_l}^{t_2+k}(a_j(x))\partial_{x_j}^{m}(u), \partial_{x_l}^{t_3}u\rangle \, dx\right.
\\
&\left. -\int_{\Omega} \langle e_le_j \partial_{x_j}^{t_1}(a_j(x))\partial_{x_j}^{t_2+k}(a_l(x))\partial_{x_l}^{m}(u),\partial_{x_j}^{t_3}u\rangle\, dx\right)\Bigg).
\end{split}
\end{equation}
We see that the first two terms in \eqref{e1} are positive. The other
terms, that we have collected in four summations, can be estimated all
in the same way and for this reason we explain how to estimate only
the integrals in the first summation. By H\"older's inequality and
the repeated application of Theorem \ref{pun},  we have for $k=0$ or $k=t_1$ 
\[
\begin{split}
\left|\int_{\Omega} \langle a_l(x)\partial_{x_l}^{t_1}(a_l(x))\partial_{x_l}^{m}(u), \partial_{x_l}^{t_2}(u)\rangle\, dx\right| &\leq \epsilon \|a_l\partial_{x_l}^m u\|^2+\frac 1{\epsilon}\|\partial_{x_l}^{t_1}(a_l)\partial_{x_l}^{t_2}(u)\|^2\\
&\leq \epsilon \|a_l\partial_{x_l}^m u\|^2+C\frac 1{\epsilon}\|\phi^{\frac 12} \partial_{x_l}^{t_2}(u)\|^2\\
&\leq \epsilon \|a_l\partial_{x_l}^m u\|^2+C'\frac 1{\epsilon}\left(\frac 2\lambda\right)^{m-t_2}\sum_{|\beta|=m}\| \partial^{\beta}(u)\|^2\\
&\leq \epsilon \|a_l\partial_{x_l}^m u\|^2+C''\frac 1{\epsilon}\left(\frac 2\lambda\right)^{m-t_2}\| u\|^2_{D^m}
\end{split}
\]
where the constants $C$, $C'$ and $C''$ depend on $\phi$, the derivatives of $a_j$'s and $m$. In the case $0<k<t_1$, we have
\[
\begin{split}
\left|\int_{\Omega} \langle\partial_{x_l}^{k}(a_l(x))\partial_{x_l}^{t_1-k}(a_l(x))\partial_{x_l}^{m}(u), \partial_{x_l}^{t_2}(u)\rangle\, dx \right| &\leq \frac 12 \|\partial_{x_l}^k (a_l)\partial_{x_l}^m u\|^2+\frac 12\|\partial_{x_l}^{t_1-k}(a_l)\partial_{x_l}^{t_2}(u)\|^2\\
&\leq \frac 12 \sup_{x\in\Omega}\left(\left| \partial_{x_l}^k a_l
  \right|^2\right)\|\partial_{x_l}^m u\|^2+C''\frac 12\left(\frac
  2\lambda\right)^{2(m-t_2)}\| u\|^2_{D^m} .
\end{split}
\]
Summing up all the previous inequalities and choosing the constants $k(m)$ and $k(m,\phi,\lambda)$ in a suitable way, we obtain
\[
\begin{split}
&\left| \sum_{l=1}^3 \sum_{|\bft'|=m\,\wedge\, t_2\leq m-1}\sum_{k=0}^{t_1}\begin{pmatrix}m\\\bft'\end{pmatrix}\begin{pmatrix}t_1\\k\end{pmatrix}\int_{\Omega} \langle\partial_{x_l}^{k}(a_l(x))\partial_{x_l}^{t_1-k}(a_l(x))\partial_{x_l}^{m}(u), \partial_{x_l}^{t_2}(u)\rangle\, dx  \right|\\
&\leq \epsilon\sum_{l=1}^3\|a_l\partial_{x_l}^m u\| +
\left(K(m)\max_{\substack{j,\,l=1,2,3\\ t=1,\dots,\,
      m}}\,\sup_{x\in\Omega}\left(\left| \partial_{x_j}^t a_l
    \right|^2\right)+\frac {K(m,\phi,\lambda)}\epsilon\right)
\|u\|^2_{D^m} .
\end{split}
\]
We note that $K(m,\phi,\lambda)$ depends on $\lambda$ through the multiplicative constant $\max\left(\left(\frac 2\lambda\right)^{2m},\,\left(\frac 2\lambda\right)^2\right)$. Analogous estimates also holds for the other summation of integrals (here the role of $k$ is played by $t_1$ and thus we have to distinguish the cases of $t_1=0$ and $t_1\neq 0$) with possibly different constants $K(m)$ and $K(m,\phi)$. Summing up all the estimates and choosing in the right way $\epsilon$, we obtain
\begin{equation}\label{nova}
\begin{split}
\Re\, b_{{\bf j}s_1}  (u,u)  &\geq s_1^2\|u\|^2_{L^2} +\frac 12 \sum_{l=1}^3\|a_l \partial_{x_l}^m u\|-M\| u\|^2_{D^m}
\\
&\geq  s_1^2 \|u\|^2_{L^2}
 + \left(\frac 12 C_T-M\right)\| u\|^2_{D^m} .
\end{split}
\end{equation}
By the hypothesis \eqref{kappaomeg} we know that
\begin{equation}\nonumber
 \frac 12 C_T-M>0
\end{equation}
thus the quadratic form $b_{{\bf j}s_1}(\cdot,\cdot)$ is coercive for every $s_1\in \mathbb{R}$.
In particular we have
\begin{equation}\label{e5}
\Re\, b_{{\bf j}s_1}(u,u)\geq s_1^2\|u\|^2_{L^2}\quad\textrm{and}\quad  \Re\, b_{{\bf j}s_1}(u,u)\geq\left(C_T-M\right)\| u\|^2_{D^m}.
\end{equation}
By the Lax-Milgram Lemma, we have that for any $w\in L^2(\Omega, \cc\otimes\mathbb H)$ there exists $u_w\in  H^m_0(\Omega,\cc\otimes\hh)$,
for $s_1\in \mathbb{R}\setminus \{0\}$ and  ${\bf j}\in \mathbb{S}$, such that
$$ b_{{\bf j}s_1}(u_w,v)= \langle w, v \rangle_{L^2},\ \ \ {\rm for \ all}\ \ v\in  H^m_0(\Omega,\cc\otimes\hh). $$
What remains to prove is the inequality \eqref{l2es2}. Applying the second of the inequalities in \eqref{nova} and observing that
$$ \|u\|_{D^m}=\sum_{\ell=1}^3\|\partial_{x_\ell}^mu\|_{L^2}^2\leq \frac{1}{C_T}\sum_{\ell=1}^3\|a_\ell\partial_{x_\ell}^m u\|_{L^2}^2 $$
we have:
\[
\begin{split}
\Re\,  b_{js_1}  (u,u) &\geq \frac 12 \sum_{l=1}^3\|a_l \partial_{x_l}^m u\|-M\| u\|^2_{D^m}
\\
&
\geq \frac 12 \sum_{\ell=1}^3\| a_\ell\partial_{x_\ell}^mu\|^2_{L^2}-\frac{M}{C_T}\sum_{\ell=1}^3\| a_\ell\partial_{x_\ell}^m u\|^2_{L^2}
\\
&
\geq \frac {C_T-2M}{2C_T}\sum_{\ell=1}^3\| a_\ell\partial_{x_\ell}^m u\|^2_{L^2}\geq C_1\|Tu\|_{L^2}^2,
\end{split}
\]
where we have set
$$
C_1:=\frac{C_T-2M}{2C_T}
$$
and this concludes the proof.
\end{proof}
\begin{remark}
We provide here two examples of domains and coefficients $a_l$'s that satisfy the hypothesis of Theorem \ref{t3}.
\begin{itemize}
\item Let $\Omega:=\{x\in\rr^3:\, |x-P|>M \}$ and $\phi(x):=\frac{e^{-\lambda|x-P|}}{|x-P|^2}$, we define $a_l(x):= K_l+e^{-\lambda |x-P|}s_l(x)$ where $s_l\in\mathcal S(\rr^3)$ and $K_l$ is a positive constant large enough for $l=1,\, 2,\,3$  ($\mathcal S(\rr^3)$ is the space of the Schwartz functions). By the properties of the Schwartz functions we have that
$$ \sup_{x\in\Omega}\left(\left|\frac{\partial^\bfb (a_l(x))}{\phi(x)}\right|\right)\leq C_{\bfb,s_l},\quad\forall \, \bfb\in\mathbb N^3\quad\textrm{and}\quad |\bfb|>0.$$
\item Let $\Omega:=\{x\in\rr^3:\, \langle x-P, v\rangle>0 \}$ and $\phi(x):=e^{-\lambda \langle x-P, v\rangle}$, we define $a_l(x):= K+e^{-\lambda \langle x-P, v\rangle}$ where $K$ is a positive constant large enough.
\end{itemize}
\end{remark}
\section{\bf The estimates for the $S$-resolvent operators and the fractional powers of $T$}\label{PRB2}
Using the estimates in Theorem \ref{t3} we can now show that the $S$-resolvent operator of $T$ decays fast enough along the set of purely imaginary quaternions.
\begin{theorem}
Under the hypotheses of Theorem \ref{t3_b} or of Theorem \ref{t3}, the operator $\mathcal Q_s(T)$
 is invertible for any
 $s={\bf j}s_1$, for $s_1\in \mathbb{R}\setminus \{0\}$ and  ${\bf j}\in \mathbb{S}$
  and the following estimate
\begin{equation}\label{ei2}
\|\mathcal Q_s(T)^{-1}\|_{\mathcal B(L^2)}\leq \frac {1}{s_1^2}
\end{equation}
 holds.
Moreover, the $\mathcal S$-resolvent operators satisfy the estimates
\begin{equation}\label{new1}
\|\mathcal S^{-1}_L(s, \,T)\|_{\mathcal B (L^2)}\leq \frac{\Theta}{|s|}\quad\textrm{and}\quad \|\mathcal S^{-1}_R(s, \,T)\|_{\mathcal B (L^2)}\leq \frac{\Theta}{|s|},
\end{equation}
for any  $s={\bf j}s_1$, for $s_1\in \mathbb{R}\setminus \{0\}$ and  ${\bf j}\in \mathbb{S}$, with a constant $\Theta$ that does not depend on $s$.
\end{theorem}
\begin{proof}
We saw in Theorem \ref{t3_b} or in Theorem \ref{t3} that for all $w\in L^2(\Omega,\cc\otimes\hh)$ there exists  $u_w\in  H^m_0(\Omega,\cc\otimes\mathbb H)$), for $s_1\in \mathbb{R}\setminus \{0\}$ and  ${\bf j}\in \mathbb{S}$, such that
$$
 b_{{\bf j}s_1}(u_w,v)= \langle w, v \rangle_{L^2},\quad {\rm for \ all}\ \ v\in H^m_0(\Omega, \cc\otimes\mathbb H).
 $$
Thus, we can define the inverse operator $\mathcal Q_{{\bf j}s_1}(T)^{-1}(w):=u_w$ for any $w\in L^2(\Omega, \cc\otimes\mathbb H)$ (we note that the range of $\mathcal Q_{{\bf j}s_1}(T)^{-1}$ is in $H^m_0(\Omega,\cc\otimes\mathbb H)$). The first inequality \eqref{e5_b} (or the first inequality in \eqref{e5} in the case $\Omega$ unbounded), applied to $u:=\mathcal Q_{{\bf j}s_1}(T)^{-1}(w)$, implies:
\begin{equation}
\begin{split}
s_1^2\|\mathcal Q_{{\bf j}s_1}(T)^{-1}(w)\|^2_{L^2} &  \leq  {\rm Re}\, b_{{\bf j}s_1}(Q_{js_1}(T)^{-1}(w),Q_{{\bf j}s_1}(T)^{-1}(w))
\\
&
\leq |b_{{\bf j}s_1}(Q_{{\bf j}s_1}(T)^{-1}(w),Q_{{\bf j}s_1}(T)^{-1}(w))|\\
& \leq |\langle w, Q_{{\bf j}s_1}(T)^{-1}(w)\rangle_{L^2}|
\\
&
\leq \|w\|_{L^2}\|Q_{{\bf j}s_1}(T)^{-1}(w)\|_{L^2}, ,
\end{split}
\end{equation}
for any $w\in L^2(\Omega, \cc\otimes\mathbb H)$. Thus,  we have
$$
 \|\mathcal Q_{{\bf j}s_1}(T)^{-1}\|_{\mathcal B(L^2)}\leq \frac {1}{s_1^2},
  \ \ \ {\rm for} \  s_1\in \mathbb{R}\setminus \{0\}\ \ {\rm and}  \ \ {\bf j}\in \mathbb{S}.
$$
The estimates \eqref{new1} follow from the estimate \eqref{l2es2_b}
(or the estimate \eqref{l2es2} in the case of $\Omega$
unbounded). Indeed, we have
\begin{equation}\nonumber
\begin{split}
 C_1\|Tu_w\|^2& \leq   {\rm Re\,} (b_{{\bf j}s_1}(u_w,u_w))
 \\
 &
 \leq |b_{{\bf j}s_1}(u_w,u_w)|
 \\
 &
 \leq |\langle w,u_w \rangle_{L^2}|
 \\
 &\leq \|w\|_{L^2}\|u_w\|_{L^2}
 \\
 &
 \overset{\eqref{ei2}}{\leq} \frac {1}{s_1^2}\|w\|^2_{L^2},
 \end{split}
 \end{equation}
  for $s_1\in \mathbb{R}\setminus \{0\}$  and $j\in \mathbb{S}$.
This estimate implies
\[
\left\| T\Q_{{\bf j}s_1}(T)^{-1}w\right\|_{L^2} = \|Tu_{w}\|_{L^2}
\leq \frac{1}{\sqrt{C_1}|s_1|} \|w\|_{L^2} ,
\]
so that we obtain
\begin{equation}\label{new4}
\left\| T \Q_{{\bf j}s_1}(T)^{-1}\right\|_{\mathcal B(L^2)} \leq \frac{1}{\sqrt{C_1}|s_1|}.
\end{equation}
In conclusion, if we set
\[
\Theta := 2\max\left\{1,\frac{1}{\sqrt{C_1}}\right\},
\]
estimates \eqref{new4} and \eqref{ei2} yield
\begin{equation}\label{eq:S_R_estimate}
\begin{split}
\left\|S_R^{-1}(s,T)\right\|_{\mathcal B(L^2)} &= \left\|(T - \overline{s}\id)\Q_{s}(T)^{-1}\right\|_{\mathcal B(L^2)} \\
&\leq \left\|T\Q_{s}(T)^{-1}\right\|_{\mathcal B(L^2)} + \left\|\overline{s}\Q_{s}(T)^{-1}\right\|_{\mathcal B(L^2)} \leq\frac{ \Theta}{|s_1|}
\end{split}
\end{equation}
and
\begin{equation*}
\begin{split}
\left\|S_L^{-1}(s,T)\right\|_{\mathcal B(L^2)} &= \left\|T\Q_{s}(T)^{-1} - \Q_{s}(T)^{-1}\overline{s}\right\|_{\mathcal B(L^2)} \\
&\leq \left\|T\Q_{s}(T)^{-1}\right\|_{\mathcal B(L^2)} + \left\| \Q_{s}(T)^{-1}\overline{s}\right\|_{\mathcal B(L^2)} \leq\frac{ \Theta}{|s_1|},
\end{split}
\end{equation*}
for any $s = {\bf j} s_1\in\hh\setminus\{0\}$.
\end{proof}

Thanks to the above results, we are now ready to establish our main statement.
\begin{theorem}\label{thm:main}
Under the hypotheses of Theorem \ref{t3_b} or of Theorem \ref{t3}, for any $\alpha\in(0,1)$ and $v\in\dom(T)$, the integral
	\[
	P_{\alpha}(T)v := \frac{1}{2\pi}\int_{-\uI\rr} s^{\alpha-1}\,ds_{\uI}\,S_{R}^{-1}(s,T) Tv
	\]
	converges absolutely in $L^2$.
\end{theorem}
\begin{proof}
	The right $S$-resolvent equation implies
	\[
	S_{R}^{-1}(s,T)Tv = sS_{R}^{-1}(s,T)v - v,\qquad \forall v\in\dom(T)
	\]
	and so
	\begin{align*}
	\frac{1}{2\pi}\int_{-\uI\rr} \left\|s^{\alpha-1}\,ds_{\uI}\,S_{R}^{-1}(s,T) Tv\right\|_{L^2}
	&\leq \frac{1}{2\pi}\int_{-\infty}^{-1} |t|^{\alpha-1} \left\| S_{R}^{-1}(-\uI t,T) \right\|_{\mathcal B(L^2)} \left\|Tv\right\|_{L^2} \,dt\\
	& \quad
	+\frac{1}{2\pi}\int_{-1}^{1} |t|^{\alpha-1} \left\| (-\uI t) S_{R}^{-1}(-\uI t,T)v - v\right\|_{L^2} \,dt\\
	& \quad
	+ \frac{1}{2\pi}\int_{1}^{+\infty} t^{\alpha-1} \left\| S_{R}^{-1}(\uI t,T) \right\|_{\mathcal{B}(L^2)} \left\|Tv\right\|_{L^2}\,dt.
	\end{align*}
	As $\alpha\in(0,1)$, the estimate \eqref{new1} now yields
	\begin{align*}
	\frac{1}{2\pi}\int_{-\uI\rr} &\left\|s^{\alpha-1}\,ds_{\uI}\,S_{R}^{-1}(s,T) Tv\right\|_{L^2}
    \\
	&\leq \frac{1}{2\pi}\int_{1}^{+\infty} t^{\alpha-1}\frac{\Theta}{t}\left\|Tv\right\|_{L^2} \,dt +
	\frac{1}{2\pi}\int_{-1}^{1} |t|^{\alpha-1} \left(|t| \frac{\Theta}{|t|} + 1\right)\|v\|_{L^2} \,dt\\
	& \quad
	+\frac{1}{2\pi}\int_{1}^{+\infty} t^{\alpha-1} \frac{\Theta}{t} \left\|Tv\right\|_{L^2}\,dt \\
	&< +\infty.
	\end{align*}
\end{proof}

\section{Concluding remarks}
We list in the following some reference associated with
 the spectral theory on the $S$-spectrum and some research directions
 in order to orientate the interested
reader in this field. Moreover, we some give references for  classical fractional problems for scalar operators.

(I) In the literature there are several non linear models that involve the fractional Laplacian and even the fractional powers
of more general elliptic operators, see for example, the books
\cite{BocurValdinoci,Vazquez}.

(II)
The $S$-spectrum approach to fractional diffusion problems used in this paper is a generalization of the
method developed by   Balakrishnan, see \cite{Balakrishnan}, to
define the fractional  powers of a real operator $A$.
In the paper \cite{64FRAC}
following the book of M. Haase, see
\cite{Haase}, has been developed
the theory on fractional powers of quaternionic linear operators, see also \cite{Hinfty,FJTAMS}.

(III) The spectral theorem on the $S$-spectrum is also an other tool
to define the fractional powers of vector operators, see \cite{ack} and for perturbation results see \cite{CCKSpert}.

(IV) An historical note on the discovery of the $S$-resolvent operators and of the $S$-spectrum
can be found in the introduction of the book \cite{CGKBOOK}.

The most important results in quaternionic operators theory based on the $S$-spectrum
and the associated theory of slice hyperholomorphic functions
are contained in the books \cite{COF,ACSBOOK,FJBOOK,CGKBOOK,ACSBOOK2,MR2752913,BOOKGS,GSSb}, for the case on $n$-tuples of operators see
\cite{JFACSS}.

(V) Our future research directions will consider the development of ideas from one and several
complex variables, such as in \cite{BKP,BKPZ,BPZ,BDP,HPR,MPR,MP2}  to the quaternionic setting.

\medskip
{\em Acknowledgements}.
The second author is partially supported by the PRIN project Direct and inverse problems for partial differential equations: theoretical aspects and applications.

\medskip
{\em Data availability}.
There are no data associate with this research.

\medskip
{\em Competing interests declaration}.
The authors declare none.

\end{document}